\documentclass[12pt, leqno]{amsart}
\usepackage{indentfirst}
\usepackage{amstext}
\usepackage{amsthm}
\usepackage{amsopn}
\usepackage{amsfonts}
\usepackage{amsmath}
\usepackage{latexsym}
\usepackage{amscd}
\usepackage{amssymb}
\usepackage{amsmath}
\usepackage[all,cmtip]{xy}

=\oddsidemargin \font\teneufm=eufm10 \font\seveneufm=eufm7
\font\fiveeufm=eufm5
\newfam\eufmfam
\textfont\eufmfam=\teneufm \scriptfont\eufmfam=\seveneufm
\scriptscriptfont\eufmfam=\fiveeufm
\def\frak#1{{\fam\eufmfam\relax#1}}
\let\goth\mathfrak

\def\gg{\goth g}

\def\GG{\frak G}

\def\gg{\goth g}
\def\gh{\goth h}

\def\gg{\goth g}

\def\1{\mbox{\bf 1}}

\def\IP{\Bbb P}

 \DeclareMathOperator{\Hom}{Hom}
\DeclareMathOperator{\Aut}{Aut} 
\DeclareMathOperator{\Out}{Out}

\DeclareMathOperator{\End}{End} 

\DeclareMathOperator{\Ad}{Ad} 
  \DeclareMathOperator{\GL}{\bf GL}

\newtheorem{theorem}{Theorem}[section]

\newtheorem{corollary}[theorem]{Corollary}

\newtheorem{lemma}[theorem]{Lemma}

\newtheorem{proposition}[theorem]{Proposition}

\theoremstyle{definition}
\newtheorem{remark}[theorem]{Remark}

\newtheorem{definition}[theorem]{Definition}

\numberwithin{equation}{section}

\def\Z{\Bbb Z}

\def\A{\mathbb A}

\def\bB{\text{\rm \bf B}}

\def\bE{\text{\rm \bf E}}
\def\bF{\text{\rm \bf F}}

\def\bG{\text{\rm \bf G}}
\def\bN{\text{\rm \bf N}}

\def\bW{\text{\rm \bf W}}
\def\bX{\text{\rm \bf X}}
\def\Mor{\text{\rm Mor}}

\def\bZ{\rm \bf{Z}}

\def\CL{{\mathcal L}}

\def\ad{\text{\rm ad}}

\def\bT{\text{\rm \bf T}}

\def\L{\mathcal{L}}

\def\bP{\text{\rm \bf P}}
\def\bY{\text{\rm \bf Y}}

\def\X{\text{\rm \bf X}}
\def\ba{\pmb a}

\def\bg{{\bf g}}

\def\bk{\mathbf k}
\def\bs{{\pmb\sigma}}

\def\bmu{{\pmb\mu }}
\def\wh{\widehat}
\def\wt{\widetilde}
\def\us{\underset}
\def\os{\overset}
\def\ol{\overline}

\def\Zar{\text{\rm Zar}}

\def\et{\text{\rm \'et}}

\def\q{\quad}
\def\vs{\vskip.3cm}

\def\noi{\noindent}

\def\2int{\mathop{2\int}\nolimits}

\def\End{\mathop{\rm  End}\nolimits}

\def\dim{\mathop{\rm dim}\nolimits}

\def\Spec{\mathop{\rm Spec}\nolimits}

\def\Hom{\mathop{\rm Hom}\nolimits}

\def\Pic{\mathop{\rm Pic}\nolimits}

\def\Aut{\text{\rm{Aut}}}
\def\Out{\text{\rm{Out}}}
\def\bAut{\text{\bf{Aut}}}
\def\bOut{\text{\bf{Out}}}

\def\resp.{\mathop{\rm resp.}\nolimits}

\def\Im{\mathop{\rm Im}\nolimits}

\def\lgr{\longrightarrow}
\def\la{\longleftarrow}

\font\math=cmmi10
\def\varpi{\hbox{\math\char'44}}
\def\simlgr{\buildrel\sim\over\lgr}

\def\pa{\S\kern.15em }

\def\un{\uppercase\expandafter{\romannumeral 1}}
\def\deux{\uppercase\expandafter{\romannumeral 2}}
\def\trois{\uppercase\expandafter{\romannumeral 3}}
\def\quatre{\uppercase\expandafter{\romannumeral 4}}
\def\cinq{\uppercase\expandafter{\romannumeral 5}}
\def\six{\uppercase\expandafter{\romannumeral 6}}

\def\gg{\goth g}
\def\et{\acute et}

\def\hfl#1#2#3{\smash{\mathop{\hbox to#3{\rightarrowfill}}\limits
^{\scriptstyle#1}_{\scriptstyle#2}}}
\def\gfl#1#2#3{\smash{\mathop{\hbox to#3{\leftarrowfill}}\limits
^{\scriptstyle#1}_{\scriptstyle#2}}}

\begin{document}

\title{Three-point Lie algebras and Grothendieck's dessins d'enfants}

\author{V. Chernousov}
\address{Department of Mathematics, University of Alberta,
    Edmonton, Alberta T6G 2G1, Canada}
\thanks{ V. Chernousov was partially supported by the Canada Research
Chairs Program and an NSERC research grant} \email{chernous@math.ualberta.ca}

\author{P. Gille}
\address{UMR 5208 du CNRS - 
Institut Camille Jordan - Universit\'e Claude Bernard Lyon 1,
43 boulevard du 11 novembre 1918,
69622 Villeurbanne cedex - France.
 \newline 
  \indent Institute of Mathematics Simion Stoilow of the Romanian Academy,
  Calea Grivitei 21, 
 RO-010702 Bucharest - Romania.}
\thanks{P. Gille was supported by the IDEI Romanian project PN-II-ID-PCE-2012-4.}
 \email{gille@math.univ-lyon1.fr}

\author{A. Pianzola}
\address{Department of Mathematics, University of Alberta,
    Edmonton, Alberta T6G 2G1, Canada.
    \newline
 \indent Centro de Altos Estudios en Ciencia Exactas, Avenida de Mayo 866, (1084) Buenos Aires, Argentina.}
\thanks{A. Pianzola wishes to thank NSERC and CONICET for their
continuous support}\email{a.pianzola@gmail.com}

\date{}




 %
 \begin{abstract}
 \noindent We define and classify the analogues of the affine Kac-Moody Lie algebras for the ring of functions on the complex projective line minus three points. The classification is given in terms of Grothendieck's dessins 
d'enfants. We also study the question of conjugacy of Cartan subalgebras for these algebras.  \\


\end{abstract}

\maketitle

\section{Introduction}

Let $\gg$ be a finite dimensional simple complex Lie algebra, and let $R= \Bbb{C}[t^{\pm 1}].$ The complex Lie algebras $\gg \otimes_{\Bbb{C}} R$ are at the heart of the untwisted affine Kac-Moody Lie algebras. There are, however, other (twisted) affine algebras. They can be realized in term of loop algebras. Relevant to us is that these loop algebras are precisely the Lie algebras over $R$ that are locally for the \'etale topology of $R$ isomorphic to $\gg \otimes_{\Bbb{C}} R$.

The ring $R$ is the ring of functions of the variety $\mathbb{P}^1_{\Bbb{C}} \setminus\{0,1\}.$ The ring $R'$ of functions of $\mathbb{P}^1_{\Bbb{C}} \setminus\{0,1, \infty\}$ has also led to fascinating class of Lie algebras 
$\gg\otimes_{\Bbb{C}} R'$; the so called three-point algebras.  
Interesting connections between three-point algebras and the Onsager algebra, which was used in the resolution of the Ising model \cite{HT},  have been described in \cite{BT}. Further references and generalizations can be found in \cite{EF} and \cite{PS}. 

If we think of the three-point algebras $\gg \otimes_{\Bbb{C}} R'$ as the analogues of the 
untwisted affine Kac-Moody Lie algebras, it is inevitable to ask if there are three-point 
analogues of the twisted affine Lie algebras. We provide a positive (and natural) answer 
to this question. 


The main aim of the paper is two-fold. First, we 
classify $\mathbb{C}$-isomorphism classes of all twisted forms  of three-point algebras. 
The resulting family of algebras is related to 
Grothendieck's dessins d'enfants.
For the details of the classification we refer
to Section 7. 

Second, we also study the (natural) question 
of conjugacy of Cartan subalgebras of three-point algebras, which is of 
importance for their (future) representation theory. In Section $8$ 
we give a cohomological description of conjugacy classes of the so called
generically maximal split tori in reductive group schemes over arbitrary base schemes 
and in Section $9$ we apply this result to conjugacy of Cartan subalgebras in twisted 
forms of three-point algebras. The main ingredient of our considerations in this part 
is the correspondence between maximal split tori of a simple group scheme  and maximal 
diagonalizable  Lie subalgebras of its Lie algebra 
given in our paper \cite[Theorem 7.1]{CGP}.

By looking at the analogue situation in the affine Kac-Moody case, we expect our work to be relevant to possible generalizations of the Wess-Zumino-Witten models and the theory of open strings (where the twisted algebras enter into the picture). Other connections to Physics as well as relevant references can be found in [Sch].
\medskip

\noindent {\bf Acknowledgement}: The  authors want to express their sincere gratitude to the referee for his/her valuable comments.

\section{Notation and conventions}
Throughout $k$ will denote an algebraically closed field of 
characteristic $0,$ and $\gg$
a finite dimensional semisimple Lie algebra over $k.$

We fix a (necessarily split) Cartan subalgebra $\gh$ of $\gg,$ and 
a base $\Delta
$ of the root system $\Phi = \Phi(\gg,\gh).$  The simply connected
(resp. adjoint) Chevalley-Demazure group over $k$ corresponding to
$\gg$ will be denoted by $\bG^{\rm sc}$  (resp. $\bG^{\ad}).$  We then
have the exact sequence of algebraic $k$-groups
$$
1\to \bmu \to \bG^{\rm sc} \os{\Ad}\longrightarrow \bG^{\ad} \to 1
$$
where $\bmu = Z(\bG^{\rm sc})$ is the centre of $\bG^{\rm sc}.$  The split
torus of $\bG^{\rm sc}$ (resp. $\bG^{\ad})$ corresponding to $\gh$ will
be denoted by $\bT^{\rm sc}$ (resp. $\bT^{\ad}).$ As usual, we denote the 
$k$-group of automorphisms of $\bG$ by $\bAut(\bG)$.

If $\bAut(\gg)$ denotes the (linear algebraic) $k$-group of
automorphisms of $\gg,$ then $\bAut(\gg)^0 \simeq \bG^{\ad}$ via
the adjoint representation. We have a split exact sequence
\begin{equation}\label{autoexact}
1 \to \bG^{\ad} \to \,\bAut(\gg)  \to \bOut
(\gg)\to 1
\end{equation}
of algebraic $k$-groups. Here $\bOut(\gg)$ is the finite
constant $k$-group corresponding to the finite (abstract) group
 of symmetries of the Dynkin diagram of $\gg.$

 In what follows we will denote the abstract groups $\bAut(\gg)(k)$ and 
 $\bOut(\gg)(k)$ by $\Aut_k(\gg)$ and $\Out_k(\gg)$ respectively.
 
 If $\X$ is a $k$-scheme and $\mathfrak{G}$ a group scheme over $\X$, we will denote 
 throughout $H^1_{fppf}(\X, \mathfrak{G})$ by $H^1(\X, \mathfrak{G}).$ Note that 
 if $\mathfrak{G}$ is smooth then $H^1_{fppf}(\X, \mathfrak{G}) = H^1_{\et}(\X, \mathfrak{G})$ 
 where $H^1_{\et}$ stands for \'etale non-abelian cohomology.



\section{Application of a result  by Harder}

\begin{proposition} \label{vanish}
 Let $C$ 
 be a smooth connected affine curve over $k$ and let $\mathfrak{G}$  
be a  semisimple
group scheme over $C.$
Then $H^1(C,\mathfrak{G})=1$.
\end{proposition}

\begin{proof}
Assume first that $\mathfrak{G}$ is simply connected. By Steinberg's theorem 
(see \cite[Theorem 11.2]{St} and also \cite{Se} Ch. III \S2, Theorem 1') we  know that 
$$\mathfrak{G}_{k(C)} = 
\mathfrak{G} \times_C \Spec\big(k(C)\big)$$ 
has a Borel subgroup (i.e., it is quasi-split). In Harder's \cite{H1}  terminology, 
this last means that $\mathfrak{G}$ is {\it rationally quasi-trivial}. By \cite[Satz 3.3]{H1} we then have  $H^1(C, \mathfrak{G})=1.$

For the general case, we consider the simply connected covering 
$f: \widetilde{\mathfrak{G}} \to \mathfrak{G}$.
Its kernel $\bmu$ is a finite \'etale $C$-group scheme and is central in 
$\widetilde{\mathfrak{G}}$.
We consider the following exact sequence  of pointed sets \cite[IV.4.2.10]{Gi}
$$
 H^1(C,\widetilde{\mathfrak{G}}) \to H^1(C,\mathfrak{G}) \to H^2(C, \bmu).
$$
But as we have seen $H^1(C,\widetilde{\mathfrak{G}}) =1$, while
  $H^2(C,\bmu)=0$ by \cite[Theorem 15.1]{M2} given that 
${\rm cd}(C)\leq {\rm dim}(C)=1$.
It follows  that
$H^1(C,\mathfrak{G})=1$.
\end{proof}

\begin{remark} If $C={\mathbb A}^1_k$ or $\GG_{m,k}$, we have furthermore  
$H^1(C,\mathfrak{G})=1$ for each reductive group $C$-scheme $\mathfrak{G}$ \cite{P1}.
  \end{remark}

\begin{corollary}\label{corvanish}
 Let $C=\Spec(R)$ be a smooth connected affine curve over $k.$  Let 
$\bG$ be a semisimple split $k$-group scheme. Then 

(i) The map $p:\bAut(\bG) \to \bOut(\bG)$ induces
a bijection of pointed sets
$$
p_*:H^1\big(C, \bAut(\bG) \big) \simlgr H^1\big(C, \bOut(\bG) \big).
$$
Similar considerations apply if we replace $\bG$ by $\gg.$ 

(ii) All $C$--forms of $\bG_C = \bG \times_k \Spec(R)$  are 
quasi-split. In particular, every semisimple group scheme over $C$ is quasi-split. \end{corollary}

\begin{proof} We have the split exact sequence of $k$ 
(or $C,$ by base change)--groups
\begin{equation}\label{autoexactgroups}
1\longrightarrow \bG^{\ad}\,\os{\text{\rm int}}
\longrightarrow \,\bAut(\bG) \overset{p}\longrightarrow
\,\bOut(\bG)\longrightarrow 1.
\end{equation}
We consider  a pinning (\'epinglage) of  $\bG$, this 
includes the choice
of  a Killing couple $(\bB,\bT)$ of $\bG$  and  defines a section
$$
h: \bOut(\bG) \to \bAut(\bG,\bB,\bT) \subset \bAut(\bG)
$$ 
(see \cite[XXIV.1]{SGA3}).

Let $\bF$ be a representative of a  class of  $H^1\big(C, \bOut(\bG)\big)$.
Consider $\bG^{\rm ad}_{h_*\bF},$ namely the semisimple $C$--group 
obtained by twisting 
$\bG^{\rm ad}$ by the $\bAut(\bG)$-torsor $h_{*}\bF.$ It is well known (see \cite{DG} II.4.5.1) that the 
classes in $H^1\big(C, \bAut(\bG)\big)$ which map to the class of $\bF$ under 
the map 
$$
p_*:H^1\big(C,\bAut(\bG)\big)\to H^1\big(C,\bOut(\bG)\big)
$$
are in one-to-one correspondence with elements in the image of
$$
H^1(C, \bG^{ad}_{h_*\bF}) \to H^1\big(C, \bAut(\bG^{ad}_{h_*\bF})\big).
$$ 
But $H^1(C, \bG^{ad}_{h_*\bF})$ vanishes by Proposition~\ref{vanish}. 
This shows that $p_*$
is injective. It is surjective because (\ref{autoexactgroups}) is split. Similar considerations apply to $\gg.$ 
Indeed if $\bG$ is of simply connected type, then $
\bAut(\bG) \simeq \bAut(\gg) $ by \cite[Exp XXIV \S 7.3]{SGA3}. This finishes the proof of (i).

(ii) For a twisted form of $\bG_C$ to be quasi-split, it suffices that the corresponding adjoint group 
be quasi-split.  That 
$\bG^{\rm ad}_{h_*\bF}$ 
is quasi-split can be explicitly seen from the fact that $h_*\bF$ stabilizes $\bB$ and $\bT$, 
so that (ii) follows from 
(i) (or, alternatively, see \cite[Satz 3.1]{H1}).
\end{proof}

\section{Affine Kac-Moody analogues of the three-point algebras}

In what follows $\gg$ denotes a simple Lie algebra over $k$. 
Kac-Moody algebras are defined by generators and relations (a l\`a
Chevalley-Serre). In the affine case, which is the most interesting outside
the classical finite dimensional setup, the algebras can be
``realized'' in terms of loop algebras (a construction that we will 
describe shortly). Let  $R = k[t^{\pm 1}]$ and $ R' = k[t^{\pm 1}]_{(t-1)}.$

The simplest case is that of the untwisted affine algebras. These
correspond to the $k$-algebras $\gg\otimes _k  R.$  The elements of
$\gg\otimes_k  R$ can be thought as morphisms (in the algebraic sense) of
$\A^1_k \setminus  \{0\}\to \A^{\dim_k(\gg)}_k$ which has a natural appeal
to many current constructions in Physics.  The same is true if we
replace $\A^1_k\setminus\{0\}=\mathbb{P}^1_k\setminus\{0,\infty  \}$ by 
$\mathbb{P}^1_k\setminus\{0, 1,\infty  \}.$  The resulting
$k$-algebras $\gg\otimes_k  R'$ are called three-point Lie algebras in the literature (see \cite{B} and also \cite{BT} for references). This raises an inevitable question:  What are the analogues, in the
case of $R'$, of the affine Kac-Moody Lie algebras which are
not of the form $\gg\otimes_k R?$ Our point of view is that the affine Kac-Moody algebras  can be thought as twisted forms of some
$\gg\otimes_k R$ in the \'etale topology of $\Spec\,(R)$ as we will explain in Theorem \ref{affineKM} below.\footnote{The term affine Kac-Moody Lie algebra is being used here to denote the corresponding loop algebra. More precisely. Let $\hat{\mathcal{L}}$ be an affine Kac-Moody Lie algebra in the sense of \cite{Kac}, and let $\mathcal{L} = [\hat{\mathcal{L}}, \hat{\mathcal{L}}]/\frak{z}$ where $\frak{z}$ is the centre of the derived algebra $[\hat{\mathcal{L}}, \hat{\mathcal{L}}]$ of $\hat{\mathcal{L}}$. It is well-known that $\mathcal{L}$ is a loop algebra. These are the algebras that we will be considering. In what follows we refer to $\CL$ as the ``derived algebra of $\hat{\CL}$ modulo its centre". }

We now review the realization of the affine algebras as loop
algebras.  

Fix once and for all a set $(\zeta  _n)_{n>0}$ of compatible
primitive $n$-roots of unity (i.e. $\zeta  ^h_{nh}=\zeta  _n).$
Having done this, the ingredient for defining loop algebras
based on $\gg$ is a  finite order  automorphism $\sigma$  of $\gg.$  For
each $i \in \mathbb{Z}$ consider the  eigenspace
$$
\gg _i :=\{x\in \gg:\,\sigma (x) =\zeta  ^{i}_mx\}
$$
where $m$ is a (fixed) period of  $\sigma$ i.e.
$\sigma  ^m = 1$.
Let
$$
R_m = k[t^{\pm \frac{1}{m}}], \; \text{\rm and}\;
R_{\infty  } =\us{{\stackrel{\longrightarrow}{m}}}\lim\; R_m.
$$
The loop algebra $L(\gg,\sigma)$ is then defined as follows
$$
L(\gg,\sigma) = \bigoplus \gg_i\otimes t^{\frac{i}{m}}
 \subset \gg\otimes_k R_m \subset \gg \otimes_k
R_{\infty  }.
$$

\vs\noi
\begin{remark}
$L(\gg,\sigma)$ is a Lie subalgebra of $\gg\otimes_k R_{\infty  }$
which does not depend on the choice of period $m$ of  $\sigma.$ Besides being an (infinite dimensional) Lie algebra
over $k,$ the loop algebra $L(\gg,\sigma)$ is also naturally a
Lie algebra over $R = k[t^{\pm 1}].$

These are the loop algebras appearing in affine Kac-Moody
theory. To connect them to non-abelian cohomology the crucial observation, which
 is easy to verify, is that
\begin{equation}\label{split}
L(\gg,\sigma)\otimes_{R} R_m \simeq \gg \otimes_k R_m \simeq (\gg \otimes_k R)\otimes_R R_m
\end{equation}
where the isomorphism is of $R_m$--Lie algebras. Thus loop algebras 
are a particular 
kind of twisted forms of the $R$-Lie algebra $\gg \otimes_k R.$ Their 
isomorphism 
classes correspond to 
a subset $H^1_{loop}\big(R,\bAut (\gg)\big)$ of $H^1
\big(R,\bAut (\gg)\big)$ 
where this last $H^1$ classifies {\it all} twisted forms of 
$\gg \otimes_k R.$ We have 
$H^1_{loop}\big(R,\bAut (\gg)\big)=H^1
\big(R,\bAut (\gg)\big)$,  namely  all
twisted forms are loop algebras. This follows  from \cite{P1}, but a more explicit result dealing with loop algebras as
Lie algebras over $k$ will given  in Theorem~\ref{affineKM} below by exploiting the fact that the centroid
of any twisted form of $\gg\otimes_k R$ is isomorphic to $R$ (see \cite{GP} 
Lemma 4.6.3\,\footnote{That $\gg$ is simple is essential. If $\gg$ is semisimple but not simple, 
then the centroid of $\CL$ is not isomorphic to $R.$}). It worth pointing out that for Laurent polynomial rings in more than one variable   $H^1_{loop}\big(R,\bAut (\gg)\big)$ does not necessarily equal $H^1
\big(R,\bAut (\gg)\big)$ in general. In other words. There exist twisted forms of $\gg \otimes_k R$ which are not multiloop algebras based on $\gg.$ See \cite{GP1} for details. 
\end{remark}

Let $\CL$ be a Lie algebra over $k.$ Recall that the centroid $C(\CL)$ of 
$\CL$ is defined by
$$
C(\CL) = \{\theta  \in \,\End_k\,(\CL):\theta  [x,y] =[x,\theta
(y)]=[\theta  (x),y]\;\text{\rm for all}\; x,y\in \CL\}.
$$
If $\CL$ is perfect, then $C(\CL)$ is commutative and we can view
$\CL$ naturally as a Lie algebra over the (commutative and unital)
ring $C(\CL).$

\begin{theorem}\label{affineKM}
For a Lie algebra $\CL$ over $k$ the following conditions are
equivalent.

\begin{description}
\item  [{\rm (i)}] $\CL$ is the derived algebra modulo its centre of an
affine Kac-Moody Lie algebra.
\item  [{\rm (ii)}] $\CL\simeq L(\gg,\sigma  )$ for some $\gg$ 
(unique up to isomorphism) and
some $\sigma  \in \,\text{\rm Aut}_k(\gg)$ of finite order.
\item [{{\rm (ii)(bis)}}] As in {\rm (ii)} but with $\sigma  \in
\text{\rm Out}_k(\gg) \subset \,\text{\rm Aut}_k(\gg).$
\item  [{\rm (iii)}] $C(\CL)\simeq R,$ and there exists an \'etale
extension $S$ of $C(\CL)$ such that $$\CL \otimes _{C(\CL)} S\simeq
\gg\otimes_k S.$$
\item  [{{\rm (iii)(bis)}}] As in {\rm (iii)} but with $S$ Galois.
\end{description}
\end{theorem}

\begin{proof} Except for uniqueness, the equivalence of (i), (ii) and 
(ii)(bis) is the realization
theorem of the affine algebras as described in Ch. 8 of \cite{Kac}. 
Given that $C(\CL) \simeq R$, 
that (ii) implies (iii)(bis)
is easy in view of (\ref{split}) and the fact that $R_m/R$ is Galois.
The rest of the equivalences hinges on Corollary \ref{corvanish}. The uniqueness of $\gg$ and that there is no loss of information in passing from $R$-algebra isomorphisms to $k$-algebra isomorphisms is explained in \cite{P1}.
\end{proof}

We now explain how cohomology enters into the picture, and how it
looks if we replace  $R = k[t^{\pm 1}]  = k[t]_t$ by $R'   = k[t^{\pm 1}]_{t-1} = k[t]_{t(t-1)}.$  
The Lie algebra
$\gg \otimes_k R'$ plays the role of the untwisted affine algebra.
The ``rest'' of the so-called three-point algebras are those whose
centroids coincide with $R',$ and which look, locally for the
\'etale topology on $\Spec(R'),$ like $\gg\otimes_k R'.$ 

\begin{definition} {\it A three-point Lie algebra is a Lie algebra $\CL$ 
over $k$ with the following two properties:

(a)  The centroid of $\CL$ is isomorphic to $R'$,

(b) There exists an \'etale extension $S'$ of $R'$ such  $\CL \otimes _{C(\CL)} S'\simeq
\gg\otimes_k S'.$}

\end{definition}

Because of Theorem \ref{affineKM}, this definition is in perfect analogy with the situation that 
one encounters in the case of affine Kac-Moody Lie algebras. 

It follows from the definition that there exists a (natural) well-defined surjective map 
from the set of $R'$-isomorphism
classes of twisted forms of $\gg\otimes_k R'$ into the set of $k$-isomorphism classes
of three-point Lie algebras which is not necessary 
injective.\footnote{This again uses  \cite{GP} Lemma 4.6.3, 
namely that the centroid on any twisted form of $\gg \otimes_k R'$ is naturally isomorphic to 
$R'$. The argument is simple. Using that $\gg$ is central, one first shows that the centroid of 
$\gg \otimes_k R'$ is naturally isomorphic to $R'$. Then the case of twisted forms is handled by 
descent considerations.} Recall that twisted $R'$-forms of $\gg\otimes_k R'$ are classified by  
$H^1\big(R',\,\bAut (\gg)\big)$.
By
considering the split exact sequence of $R'$--groups corresponding to 
(\ref{autoexact})
and passing to cohomology we obtain
\begin{equation}
H^1(R',\bG^{\ad}) \to H^1\big(R',\,\bAut(\gg)\big) \to
H^1\big(R',\,\bOut (\gg)\big) \to 1.
\end{equation}
By Corollary \ref{corvanish} we in fact have a bijection of pointed sets
\begin{equation}
H^1\big(R',\,\bAut(\gg)\big) \to
H^1\big(R',\,\bOut (\gg)\big).
\end{equation}
Moreover by \cite{SGA1} XI \S5
\begin{equation}
H^1\big(R',\bOut(\gg)\big)\simeq
\,\Hom_{cont}\big(\pi_1(R'),\text{\rm Out}_k(\gg)\big)/\text{\rm conjugation}
\end{equation}
where $\pi_1(R')$ is the algebraic fundamental group at the
geometric point 
$$
\ol x =\,\Spec\,\ol{k(t)} \to\,\Spec\, (R')
$$
 and $\text{\rm Out}(\gg)$ acts on
$\Hom_{cont}\big(\pi_1(R'),\text{\rm Out}(\gg)\big)$ by conjugation.

\begin{remark}\label{comparison} By \cite{Sz} 4.6.12 or the Comparison 
Theorem $\pi_1(R')\simeq \widehat{\mathbb{Z}*\mathbb{Z}}.$  If
we fix generators $e_1$ and $e_2$ for the two copies of $\mathbb{Z},$ then
$H^1\big(R',\bOut_k(\gg)\big)$ can be thought as conjugacy classes
of pairs of elements of $\Out_k (\gg).$
\end{remark}

\begin{remark}\label{dessin}It is well known  that the abstract group 
${\rm Out}_k({\gg})$ is the symmetric group $S_n$ with $n = 1,2,3.$ 
Furthermore, if the pair of elements of $S_n$ assigned to $e_1 $ and $e_2$ 
generate a subgroup in $S_n$ acting transitively on $\{1,\ldots,n\}$ 
then the conjugacy class of this pair corresponds to a dessin d'enfant of 
degree $n$ and conversely (see below for details and references). As we shall see the classification of 
three-point algebras, as Lie algebras over $R'$, is entirely given 
in terms of dessins.
\end{remark}

\section{Dessins d'enfants}

In this section we briefly recall the definition of dessins d'enfants and their properties.
For more details and applications we refer to the surveys 
\cite{LZ}, \cite{Sch}, \cite{W}. 

A dessin d'enfants is a bipartite connected graph $\Gamma$ which is embedded into
an oriented closed (connected) topological surface $X$ such that it fills the surface, i.e.
$X\setminus\Gamma$ is a union of open cells. Two dessins d'enfants $(X_1,\Gamma_1)$
and $(X_2,\Gamma_2)$ are called equivalent if there exists a homeomorphism
$f:X_1\to X_2$ such that $f(\Gamma_1)=\Gamma_2$.

Such graphs appear in a nice way to describe all possible coverings $\beta:X\to 
\mathbb{P}^1(\mathbb{C})$
from a closed (connected) Riemann surface $X$ to the Riemann sphere $\mathbb{P}^1(\mathbb{C})$
which are ramified at most over the points $0,1$ and $\infty$. Such pair
$(X,\beta)$ is called a Belyi morphism. The remarkable fact, known as Belyi's theorem, is that a projective smooth connected
curve $X$ over $\mathbb{C}$ is defined over $\overline{\mathbb{Q}}$ if and only if 
there exists a finite covering $X\to \mathbb{P}^1$ which is unramified outside $0,1$ and $\infty$.

Given a Belyi pair $(X,\beta)$ we associate a bipartite graph
$\Gamma$ on $X$ as follows. We may identify $\mathbb{P}^1(\mathbb{C})$ with
$\mathbb{C}\cup \{\infty\}$. Let $I=[0,1]$ be the closed segment on the real line $\mathbb{R}$.
Then its preimage $\beta^{-1}(I)$ is a connected graph on $X$. Its vertices are the preimages 
of $0$ and $1$. We may colour all preimages of $0$ with one colour (say white)
and all preimages of $1$ with another colour (say black). One checks that $\beta^{-1}(I)$
fills $X$ and hence it is a dessin d'enfants.

It is a striking fact  that the above correspondence is one-to-one. 
\begin{theorem}\label{dessins}
A Belyi pair $(X,\beta)$ of degree $d$
is uniquely determined up to equivalence by
\begin{itemize}
\item a dessin d'enfants with $d$ edges up to equivalence;

 \item a monodromy map $\alpha: F_2\to S_d$, i.e. a transitive action of the free group 
 $F_2$ on two generators on a $d$-element set $\{1,2,\ldots,d\}$,
 up to conjugation;
 
  \end{itemize}
 \end{theorem}
\begin{proof} See \cite{HS}. 
\end{proof}

\section{Classification of three-point algebras over $R'$}

\begin{theorem}\label{R'classification}
The classification of affine three-point Lie algebras over $R' =
k[t]_{t(t-1)}$ is as follows.

\begin{description}
\item  [{\rm (i)}] If $\gg$ is of type $A_1,B_\ell\,\, (\ell \geq 2),C_\ell\,\, (\ell \geq 3),G_2,
F_4,E_7$ or $E_8$ then all three-point affine algebras are trivial
i.e. isomorphic to $\gg \otimes_k R'.$

\item  [{\rm (ii)}] If $\gg$ is of type $A_\ell$ $(\ell >1),$ $D_\ell$
$(\ell >4)$ or $E_6$ there are four isomorphism classes of three-point
affine algebras. The trivial algebra, and three ``quadratic" algebras.


\item  [{\rm (iii)}] If $\gg$ is of type $D_4$ there are eleven {\rm 3}-pointed
affine algebras. The trivial algebra, three ``quadratic" algebras, and seven ``trialitarian" algebras of which four are ``cyclic cubic" and three are ``non-cyclic cubic"--algebras.
\end{description}
\end{theorem}

\begin{proof} By general considerations, as explained in \S4, three-point Lie algebras 
as $R'$-algebras are classified by $H^1\big(R',\,\bAut(\gg)\big)$, and as we have seen 
the natural map $H^1\big(R',\,\bAut(\gg)\big) \to
H^1\big(R',\,\bOut (\gg)\big)$ is bijective. In view of 
Remarks \ref{comparison} and  \ref{dessin}, the classification is thus given by computing 
conjugacy classes of pairs of elements of the symmetric groups $S_n$ where $n = 1,2,3$ in 
cases (i), (ii) and (iii) respectively and, if the resulting action is transitive,  we provide 
the corresponding  dessin.

(i) It is obvious that there is only one conjugacy class whose dessin is just the interval $I$, 
corresponding to the trivial cover of the Riemann sphere.

\begin{table}[!ht]
\renewcommand{\arraystretch}{2}
\begin{tabular}{|c|c|c|c|c|c|c|}
\hline
\parbox{2.5cm}{\centering$S_1$ pair\\
representative}&Dessin&$n_0$&$n_1$&$n_\infty$&$g$\\
\hline
$(1,1)$&$\begin{xy}\ar@{-} (0,0)*{\circ};(10,0)*{\bullet};\end{xy}$&1&1&1&0\\
\hline
\end{tabular}
\end{table}
\noindent
Here (and everywhere below) $n_0,n_1$ and $n_{\infty}$ are the number of points over
$0,1$ and $\infty$ respectively under the corresponding cover
$\beta:X\to \mathbb{P}^1(\mathbb{C})$ and $g$ is the genus of $X$.
Note that if $\beta:X\to\mathbb{P}^1$ has degree $n$ then the genus $g$ 
of $X$ is equal to $(n-n_0-n_1-n_{\infty}+2)/2$ and hence $n_{\infty}$ is determined uniquely 
by $n_0,n_1$ and $g$.
 
(ii) Let $r = (1 2)$ be the generator of $S_2.$ There are four conjugacy 
classes of pairs of elements of $S_2$. The trivial class $(1,1)$ and 
three transitive classes with representatives and dessins given by the following table.
\begin{table}[!ht]
\renewcommand{\arraystretch}{2}
\begin{tabular}{|c|c|c|c|c|c|c|}
\hline
\parbox{2.5cm}{\centering$S_2$ pair\\
representative}&Dessin&$n_0$&$n_1$&$n_\infty$&$g$\\
\hline
$(1,r)$&$\begin{xy}(0,0)*{\circ}="a",(10,0)*{\bullet}="b",(20,0)*{\circ}="c",\ar@{-} "a";"b" \ar@{-} "b";"c"\end{xy}$&2&1&1&0\\
\hline
$(r,1)$&$\begin{xy}(0,0)*{\bullet}="a",(10,0)*{\circ}="b",(20,0)*{\bullet}="c",\ar@{-} "a";"b" \ar@{-} "b";"c"\end{xy}$&1&2&1&0\\
\hline
$(r,r)$&$\begin{xy}(0,0)*{\circ}="a",(10,0)*{\bullet}="b",
\ar@{-}@/^0.3pc/ "a";"b" 
\ar@{-}@/_0.3pc/ "a";"b" \end{xy}$&1&1&2&0\\
\hline
\end{tabular}
\end{table}

 Each conjugacy class in fact consists of a single element and the corresponding covers are 
 given by the \'etale extension
 $S/R'$ where $S=R'(\sqrt{t-1}),\,R'(\sqrt{t})$ and $R'(\sqrt{t(t-1)})$ respectively. 

\medskip

(iii) Let $r = (12),$ $s =(23)$ and $c = sr = (123)$. There are eleven conjugacy classes of pairs of elements of $S_3$. 
The trivial class $(1,1),$  three ``quadratic" classes  $(1,r), (r,1), (r,r)$ (whose corresponding 
algebras are obtained as in (ii) above), and seven transitive classes represented by

\begin{table}[!ht]
\renewcommand{\arraystretch}{2}
\begin{tabular}{|c|c|c|c|c|c|c|}
\hline
\parbox{2.5cm}{\centering$S_3$ pair\\
representative}&Dessin&$n_0$&$n_1$&$n_\infty$&$g$&\parbox{2cm}{\centering Trialitarian\\
Type}\\
\hline
$(1,c)$&$\begin{xy}(0,0)*{\bullet}="a",(0,-5)*{\circ}="b",(5.2,3)*{\circ}="c",(-5.2,3)*{\circ}="d",\ar@{-} "a";"b" \ar@{-} "a";"c" \ar@{-} "a";"d"\end{xy}$&3&1&1&0&cyclic\\[3ex]
\hline
$(c,1)$&$\begin{xy}(0,0)*{\circ}="a",(0,-5)*{\bullet}="b",(5.2,3)*{\bullet}="c",(-5.2,3)*{\bullet}="d",\ar@{-} "a";"b" \ar@{-} "a";"c" \ar@{-} "a";"d"\end{xy}$&1&3&1&0&cyclic\\[3ex]
\hline
$(c,c^2)$&$\begin{xy}(0,0)*{\circ}="a",(10,0)*{\bullet}="b",\ar@{-}@/^0.4pc/ "a";"b" \ar@{-}@/_0.4pc/ "a";"b" \ar@{-} "a";"b"\end{xy}$&1&1&3&0&cyclic\\
\hline
$(c,c)$&\mbox{$\begin{xy}(0,0)*{\circ}="a",(10,0)*{\bullet}="b",\ar@{-}@/^1pc/ "a";"b" \ar@{-}@/_1pc/ "a";"b" \crv{"a"&(19,6)&"b"}\end{xy}$}&1&1&1&1&cyclic\\[1ex]
\hline
$(r,s)$&$\begin{xy}(0,0)*{\circ}="a",(10,0)*{\bullet}="b",(20,0)*{\circ}="c",(30,0)*{\bullet}="d", \ar@{-} "a";"b" \ar@{-} "b";"c" \ar@{-} "c";"d"\end{xy}$&2&2&1&0&non-cyclic\\
\hline
$(r,c)$&$\begin{xy}(0,0)*{\circ}="a",(10,0)*{\bullet}="b",(20,0)*{\circ}="c",
\ar@{-} "a";"b" \ar@{-}@/^0.4pc/ "b";"c" \ar@{-}@/_0.4pc/ "b";"c"\end{xy}$&2&1&2&0&non-cyclic\\
\hline
$(c,r)$&$\begin{xy}(0,0)*{\bullet}="a",(10,0)*{\circ}="b",(20,0)*{\bullet}="c",
\ar@{-} "a";"b" \ar@{-}@/^0.4pc/ "b";"c" \ar@{-}@/_0.4pc/ "b";"c"\end{xy}$&1&2&2&0&non-cyclic\\
\hline
\end{tabular}
\end{table}
Here the column Trialitarian Type refers to the Galois groups of the corresponding
covers $(X,\beta)$ which are isomorphic to $\mathbb{Z}/3\mathbb{Z}$ or $S_3$. For explicit descriptions of the corresponding 
covers  we refer to \cite{Z}.
\end{proof}

\section{Classification of three-point algebras over $k$}

Just as in the case of Kac-Moody Lie algebras, in infinite dimensional Lie theory one 
is interested in viewing three-point algebras as algebras 
over $k$ and not over $R'$. 
Clearly, if two three-point algebras are isomorphic over $R'$ 
then they are isomorphic as 
$k$-algebras as well.
However the converse is not true. Before proceeding to their 
classification over $k$ we first state and prove a useful criterion 
for studying the $R$ vs $k$--isomorphism question for twisted forms of $\gg \otimes_k R$ for an arbitrary $R.$

Consider then an {\it arbitrary} (commutative, associative and unital) $k$-algebra $R$ and the set of all $R$-isomorphism classes 
of forms of
the $R$-Lie algebra $\mathfrak{g}\otimes_k R.$ These are
 classified 
by the pointed set $H^1\big(R,{\bAut(\mathfrak{g})}\big)$. We have a natural action
of the abstract group $\Gamma={\rm Aut}_k(R)$ on the set of isomorphism classes of
$\bAut(\mathfrak{g})$-torsors. Namely, every element 
$\gamma\in\Gamma$ 
gives rise to 
the scheme automorphism $\gamma_*:{\rm Spec}(R)\to {\rm Spec}(R)$. 
If $\mathcal{T}\to 
{\rm Spec}(R)$ is an $\bAut(\mathfrak{g})$-torsor the base change $\gamma_*$ 
produces a scheme $^{\gamma}\mathcal{T}=\mathcal{T}\times_{\gamma_*} {\rm Spec}(R)$ which is obviously
an $\bAut(\mathfrak{g})$-torsor. Thus $\gamma$ yields a natural bijection 
$$
\gamma_*:H^1\big(R,\bAut(\mathfrak{g})\big)\to H^1\big(R,\bAut(\mathfrak{g})\big).
$$
Similarly, $\Gamma$ acts on the set of isomorphism classes of twisted $R$-forms of
 the simply connected group $\bG^{sc}$ attached to $\mathfrak{g}$ and their Lie algebras.
 

On the level of Lie algebras over $R$ the action of $\gamma\in\Gamma$ can be 
seen as follows.
Let $\mathcal{L}$ be an $R$-form of $\mathfrak{g} \otimes_k R$.
Then we define $^{\gamma}\mathcal{L}$ to be $\mathcal{L}$ (as a set) 
with the same Lie 
bracket structure, but the $R$-module structure is given by composing the automorphism $\gamma:R\to R$ with the standard action of $R$ on $\mathcal{L}$. 
Thus for all $x,y\in\mathcal{L}={^{\gamma}\mathcal{L}}$ we have
$$
[x,y]_{{^{\gamma}\mathcal{L}}}=[x,y]_{\mathcal{L}}
$$ and for all $r\in R$
we have 
$$
r\cdot_{^{\gamma}\mathcal{L}} x=\gamma(r)\cdot_{\mathcal{L}}x.
$$  
In particular, it follows that the identity mapping $f_{\gamma}:
{^{\gamma}\mathcal{L}}\to \mathcal{L}$ is an isomorphism of $k$-Lie algebras.
\begin{proposition} Let $\mathcal{L}$ and $\mathcal{L}'$ be $R$-forms
of $\mathfrak{g} \otimes_k R$. Then $\mathcal{L}$ and $\mathcal{L}'$ are isomorphic
as $k$-Lie algebras if and only if there exists $\gamma\in\Gamma={\rm Aut}_k(R)$ 
such that $\mathcal{L}'\simeq{^{\gamma}\mathcal{L}}$ over $R$.
\end{proposition}
\begin{proof} This follows essentially from the fact  that the centroids
$C(\mathcal{L})$ and $C(\mathcal{L}')$ of $\mathcal{L}$ and $\mathcal{L}'$ 
are naturally isomorphic to $R$, as we now explain. For $r \in R$ consider the 
homothety $\chi_r \in \End_k(\L)$ given by $\chi_r : x \mapsto rx.$ Clearly $\chi_r \in C(\L).$ 
This gives a natural morphism of associative and commutative $k$-algebras $R \to C(\L)$ which 
is known to be an isomorphism (see 
\cite[Lemma 3.4]{P2} and 
\cite[Lemma 4.6]{GP}).  

We now turn to the proof of the Proposition. Assume that $f:\mathcal{L}\to\mathcal{L}'$ is a $k$-algebra isomorphism. It is immediate to verify that the map $C(f) : C(\L') \to \End_k(\L)$ defined by $\chi \mapsto f^{-1}\circ \chi \circ f$ is in fact a $k$-algebra isomorphism of $C(\L')$ onto $C(\L).$ Under our identification $C(\L') \simeq R \simeq C(\L)$ we thus have $\gamma := C(f)  \in {\rm Aut}_k(R).$   

The reader can easily verify from the definition that for the isomorphism $f_{\gamma}:
{^{\gamma}\mathcal{L}}\to \mathcal{L}$ above one has $C(f_\gamma) = \gamma^{-1}.$ Thus, if we consider the composition 
$$
f' : {^{\gamma}\mathcal{L}}\stackrel{f_{\gamma}}{\longrightarrow} \mathcal{L}
\stackrel{f}{\longrightarrow} \mathcal{L}'
$$
then $C(f') = 1.$ But this means that $f'$ is $R$-linear.

The converse is clear since we have already observed that $f_{\gamma}:
{^{\gamma}\mathcal{L}}\to \mathcal{L}$ is an isomorphism of $k$-Lie algebras.\end{proof}

\begin{remark}\label{Rvsk} The above Proposition has nothing to do with Lie algebras. We could replace $\gg$ by any finite dimensional central simple algebra over $k.$
\end{remark}
 
We now come back to the ring $R'$. For such a ring we have ${\rm Aut}_k(R')\simeq S_3$. Indeed,
an arbitrary $k$-automorphism $\gamma:R'\to R'$ induces an automorphism ${\rm Spec}(R')\to
{\rm Spec}(R')$ and hence an automorphism of the fraction field $k(t)$ of $R'$. On the 
other hand, by \cite[Chapter 1, Exercise 6.6]{Hart} one has ${\rm Aut}_k \,k(t)\simeq 
{\rm Aut}(\mathbb{P}^1_k)\simeq {\rm PGL}_2(k)$ and thus we obtain a mapping from ${\rm Aut}_k(R') $ 
into the permutation group $S_3$ of the points $\{0,1,\infty\}$ on $\mathbb{P}^1_k$.
Also, it is known \cite[Chapter IV, Exercise 2.2]{Hart} that any permutation of 
$\{0,1,\infty\}$
is induced by a unique automorphism of $\mathbb{P}^1_k$. This induces an 
automorphism of  $\mathbb{P}^1_k \setminus \{0,1,\infty\},$  hence  of $\Spec(R')$.

We now pass to the classification of $k$--isomorphism classes of 
three-point Lie algebras.
\begin{theorem}
The classification of three-point Lie algebras over $k$ is as follows.

\begin{description}

\item  [{\rm (i)}] If $\gg$ is of type $A_1,B_\ell\,\, \ell \geq 2,C_\ell\,\, 
\ell \geq 3,G_2,
F_4,E_7$ or $E_8$ then all three-point Lie algebras are trivial
i.e. isomorphic to $\gg \otimes_k R'.$

\item  [{\rm (ii)}] If $\gg$ is of type $A_\ell$ $\ell >1,$ $D_\ell$
$\ell >4$ or $E_6$ there are two isomorphism classes of three-point Lie algebras. The trivial algebra, and one quadratic algebra arising for the 
three quadratic algebras over $R'$. More precisely, the three non-isomorphic 
quadratic algebras over $R'$ are isomorphic as $k$--algebras.


\item  [{\rm (iii)}] If $\gg$ is of type $D_4$ there are five three-point Lie algebras.

\begin{enumerate}
\item  The trivial algebra.

\item   One quadratic algebra arising from the three quadratic algebras over $R'.$

\item  One cyclic cubic algebra arising from the three cyclic cubic algebras over 
$R'$ corresponding to dessins of genus $0.$

\item  One cyclic cubic algebra corresponding to the cyclic cubic $R'$-algebra whose 
dessin is of genus $1.$

\item  One non-cyclic cubic algebra arising from the three non-cyclic cubic 
algebras over $R'.$
\end{enumerate}
\end{description}
\end{theorem}

\begin{proof}
(i) Follows for Theorem \ref{R'classification}(i) because there exists a unique $R'$-form of
$\mathfrak{g}\otimes_k{R'}$. Indeed, in this case we have $\Out(\mathfrak{g})=1$.

(ii) and (iii): The problem is equivalent to classifying the set of $R'$-forms of
$\mathfrak{g}\otimes_k{R'}$ up to the action of $\Aut_k(R')=S_3$. Let $\mathcal{L}$
be an $R'$-form of $\mathfrak{g}\otimes_k{R'}$ and let $\gamma\in S_3$. The form
$\mathcal{L}$ corresponds to some class $[\xi]\in H^1\big(R',\bOut(\mathfrak{g})\big)$
and hence to a homomorphism $F_2\to S_d$, $d=2,3$, which we may assume to be
transitive (if it is not, we replace $S_d$ by $S_{d'}$ with $d' < d.$) 
By Theorem~\ref{dessins} this leads us to a unique Belyi pair 
$\beta:X\to \mathbb{P}^1(\mathbb{C})$.

Let $\gamma\in S_d$. The form $^{\gamma}\mathcal{L}$ corresponds to the class of the torsor  
$^{\gamma}\xi$
and hence to another Belyi pair $^{\gamma}\beta: {^{\gamma}X}\to \mathbb{P}^1(\mathbb{C})$.
This new Belyi pair $(^{\gamma}X,{^{\gamma}\beta})$ can be explicitly described as follows:
as we noticed before the permutation $\gamma$ of the set $\{0,1,\infty\}$ gives rise 
to a unique automorphism $\gamma_*:\mathbb{P}^1\to\mathbb{P}^1$ and then we can take
$^{\gamma}X$ to be equal to $X\times_{\gamma_*} \mathbb{P}^1$ and
$^{\gamma}\beta$ to be a natural projection $^{\gamma}X=X\times_{\gamma_*} 
\mathbb{P}^1\to \mathbb{P}^1$. As a consequence of this construction we conclude that
the numbers $n_0,n_1,n_{\infty}$ for the pair $(^{\gamma}X,{^{\gamma}\beta})$
are obtained from those of  the Belyi pair $(X,\beta)$ by permuting
them with the use of $\gamma$. Looking at the tables $2$ and $3$ we observe that the 
cover is fully determined by the number of preimages of 0, 1, and $\infty$. In case (ii) three 
of the twisted forms of $\mathfrak{g} \otimes_k {R'}$ presented
in table $2$ 
are in the same $S_3$-orbit, and that in case (iii) there are five $S_3$-orbits
as claimed in the theorem.
\end{proof}


\begin{remark} {\rm (i)}  An   \'etale extension  of  $R'$ of degree $2$ is  isomorphic,
up to $\text{\rm Aut}_k\,(R')$-base change, to either $R' \times R'$ or
$R'[\sqrt{t}]$.

\noindent {\rm (ii)}
An \'etale extension of $R'$ of degree $3$ is isomorphic,  up to
$\text{\rm Aut}_k\,(R')$-base change, to one (and only one) of the algebras of 
the following  list:

\smallskip

(a) $R' \times R' \times R'$;

\smallskip

(b) $R'[\sqrt{t}] \times R'$ (the quadratic case);

\smallskip

(c) $R'[\sqrt[3]{t}]$ (the cyclic cubic case of genus $0$);

\smallskip

(d) $R'[\sqrt[3]{t(t-1)}]$ (the cyclic cubic case of genus $1$);

\smallskip

(e) $R'[X]/(X^3+3X^2 -4t)$ (the non-cyclic cubic case of genus $0$).

\end{remark}

\begin{remark}\label{alldessins} Let $n > 0$ and consider the semisimple Lie algebra $\gg_n = \mathfrak{sl}_2 \times \cdots \times \mathfrak{sl}_2$ (n copies). Then $\bOut (\gg)(k) = S_n.$ As in \S4 we have bijections
$$
H^1\big(R',\,\bAut(\gg)\big) \to
H^1\big(R',\,\bOut (\gg)\big)
$$

and

$$
H^1\big(R',\bOut(\gg)\big)\simeq
\,\Hom_{cont}\big(\pi_1(R'),\text{\rm Out}_k(\gg)\big)/\text{\rm conjugation}.
$$
\medskip

By Theorem \ref{dessins} we obtain a correspondence between dessins d'enfants of degree 
$\leq n$ and 
isomorphisms classes of twisted forms of the $R'$-Lie algebra $\gg_n \otimes_k R'.$ Note, however, 
that the methods of this section cannot be applied to the classification of these algebras over $k$. 
The reason is that for $n > 1$, the natural map from $R'$ to the centroid of $\gg_n$ is not an isomorphism.

We use this opportunity to clarify a misunderstanding regarding loop algebras. Assume $\gg$ is semisimple and that $\sigma$ is an automorphism of finite order of $\gg$. It is not true that the isomorphism class of the loop algebra $L(\gg, \sigma),$ as a Lie algebra over $k,$ depends only on the ``outer part" of $\sigma.$ The assertion is true up to $R$-isomorphism classes as shown in \cite{P1}, but not up to $k$-isomorphism.

Here is a concrete example. Let $\sigma$ be the automorphism of $\gg_2$ that switches the two copies of $\mathfrak{sl}_2$. The outer part of $\sigma$ is the diagram automorphism that switches the two disjoint nodes of the corresponding Coxeter-Dynkin diagram. The Lie algebras $L(\gg_2, {\rm id})$ and $L(\gg_2, \sigma)$ are not isomorphic as $R$-Lie algebras, but are isomorphic as $k$-Lie algebras.

\end{remark}

\section{Cohomological description of conjugacy classes of split tori}

Let $\mathfrak{G}$ be a reductive group scheme  
 over a base
scheme $X$ and let $\mathfrak{S}$ be a split subtorus of $\mathfrak{G}$.
According to \cite[XI.5.9]{SGA3}, the fppf sheaf 
$\bN_\mathfrak{G}(\mathfrak{S})/\bZ_\mathfrak{G}(\mathfrak{S})$ is representable
by a $X$--group scheme $\bW(\mathfrak{S})$ called the Weyl group scheme.
The action of $\bN_\mathfrak{G}(\mathfrak{S})$ on $\mathfrak{S}$ gives rise
to a monomorphism $i :\bW(\mathfrak{S}) \to \bAut(\mathfrak{S})$
which is an open immersion so that $\bW(\mathfrak{S})$ is quasi-finite, separated and 
\'etale over $X$ ({\it loc. cit.}).

\begin{lemma} \label{critical} $\bW(\mathfrak{S})$ is a finite group scheme over $X$ which is Zariski-locally constant.
\end{lemma}

\begin{proof} We shall prove firstly that $\bW(\mathfrak{S})$ is finite over $X$.
The centralizer $ \mathfrak{L}:=\bZ_\mathfrak{G}(\mathfrak{S})$ is a Levi subgroup of   a 
parabolic subgroup $\mathfrak{P}$ of $\mathfrak{G}$ \cite[XXVI.6.2]{SGA3}. We denote by $\mathfrak{R}$ the  
radical of $\mathfrak{L},$ namely its 
maximal central subtorus  \cite[XXII.4.3.6]{SGA3}. Observe that $\mathfrak{L}$ is a critical 
$X$--subgroup of $\mathfrak{G}$ \cite[XXVI.1.13]{SGA3}, i.e. we have
$$
\mathfrak{L}= \bZ_{\mathfrak{G}}(\mathfrak{R}).
$$
The Weyl group scheme $\bW(\mathfrak{R})$ is finite according to 
\cite[XXII.5.10.9]{SGA3}. We shall link $\bW(\mathfrak{R})$ and $\bW(\mathfrak{S})$
by using the fact that $\bN_\mathfrak{G}(\mathfrak{L}) = \bN_\mathfrak{G}(\mathfrak{R})$.
Since $ \mathfrak{L} =\bZ_\mathfrak{G}(\mathfrak{S})$, 
we  have  $
\bN_\mathfrak{G}(\mathfrak{S}) \subset \bN_\mathfrak{G}(\mathfrak{L}) = \bN_\mathfrak{G}(\mathfrak{R})$
and both  are closed subgroups of $\mathfrak{G}$ \cite[XI.5.9]{SGA3}. 
In particular $  \bN_\mathfrak{G}(\mathfrak{S})$ is a closed subgroup of
$\bN_\mathfrak{G}(\mathfrak{R})$.
By moding out by $\mathfrak{L}= \bZ_{\mathfrak{G}}(\mathfrak{S})= 
\bZ_{\mathfrak{G}}(\mathfrak{R})$, 
we get a closed immersion of $X$--groups $\bW(\mathfrak{S}) \to \bW(\mathfrak{R})$ 
(it is a closed immersion
by fppf descent \cite[Prop 2.7.1(xii)]{EGAIV}) which is in particular a finite morphism. 
Since  $\bW(\mathfrak{R})$ is  finite over 
$X$, we conclude that  $\bW(\mathfrak{S})$ is finite over $X$.

From what has been shown heretofore  we get that 
$\bW(\mathfrak{S})$ is finite \'etale over $X.$ Consider the 
 open immersion $i :\bW(\mathfrak{S}) \to \bAut(\mathfrak{S})$.
To show that $\bW(\mathfrak{S})$ is Zariski-locally constant, we can obviously assume that
$X$ is affine, hence that is $\bW(\mathfrak{S})$ as well.

We use now that  $\mathfrak{S}$ is split, that is $\mathfrak{S}\simeq \GG_m^r$ for some $r \geq 0$.
According to \cite[VII.1.5]{SGA3}, the fppf sheaf $\bAut(\mathfrak{\GG_m^r})$ is representable  
by $\GL_r(\Z)_{X}$ where 
$$\GL_r(\Z)_{X}= \bigsqcup_{\sigma \in {\rm GL}_r(\Z)} X_\sigma$$ stands for 
the constant $X$--group scheme attached
to the abstract group ${\rm GL}_r(\Z)$. The morphism $i$ induces a decomposition in open subsets
$$\bW(\mathfrak{S})  = \bigsqcup_{\sigma \in {\rm GL}_r(\Z)} \bW(\mathfrak{S})_\sigma$$
where $ \bW(\mathfrak{S})_\sigma$ is the inverse image of $X_{\sigma}$ under $i$.
Since $\bW(\mathfrak{S})$ is affine, almost all $\bW(\mathfrak{S})_\sigma$ are 
empty so that there exists a finite subset $\Sigma$ of ${\rm GL}_r(\Z)$ such that
$\bW(\mathfrak{S})  = \bigsqcup_{\sigma \in \Sigma} \bW(\mathfrak{S})_\sigma$.
Since $i$ is an open immersion, each morphism $\bW(\mathfrak{S})_\sigma \to X$ is a clopen $X$--immersion.

For each $\sigma \in \Sigma$, we denote by $U^+_{\sigma}$ the (isomorphic) image of  $\bW(\mathfrak{S})_\sigma$
in $X$ and by $U^{-}_{\sigma} = \bigcup\limits_{\tau \not = \sigma} U^+_\tau$.
For each function $\epsilon : \Sigma \to \{+,-\}$, we define the open subset 
 $U^\epsilon= \bigcap \limits_{\sigma \in \Sigma} U^{\epsilon(\sigma)}_{\sigma}$ of $X$.
Then $(U^{\epsilon})_{\epsilon}$ (for $\epsilon$ running over the functions  $\Sigma \to \{+,-\}$)
is an open cover of $X$ 
such that $\bW(\mathfrak{S}) \times_X U^\epsilon$ is a finite constant $U^\epsilon$-scheme for each such function 
$\epsilon$. We conclude that $\bW(\mathfrak{S})$ is a finite $X$--group scheme which is locally constant for
the Zariski topology.
\end{proof}


The fppf sheaf $\bY= \mathfrak{G}/\bN_\mathfrak{G}(\mathfrak{S})$ is representable 
\cite[XI.5.3.bis]{SGA3}
and  for each scheme $X'$ over $X$, the  elements of the set $\bY(X')$ are in one-to-one 
correspondence with 
$X'$-subtori $\mathfrak{S}'$ of $\mathfrak{G}_{X'}$ which are fppf-locally conjugate to 
$\mathfrak{S}_{X'}$.
Since $\bN_\mathfrak{G}(\mathfrak{S})$ is smooth, locally conjugation with respect 
to \'etale 
topology is equivalent to conjugacy  in the fppf topology.

The orbits of $\mathfrak{G}(X)$ on $\bY(X)$ can be described by means of the
exact sequence $$
\bY(X) \stackrel{\phi}{\longrightarrow} H^1\big(X, \bN_\mathfrak{G}(\mathfrak{S})\big)
\longrightarrow H^1(X, \mathfrak{G})
$$
arising from an exact sequence
$$
1\longrightarrow \bN_{\mathfrak{G}}(\mathfrak{S})\longrightarrow 
\mathfrak{G}\longrightarrow \bY
\longrightarrow 1.
$$
More precisely,
 we have a natural bijection \cite[III.3.2.4]{Gi}
$$
\mathfrak{G}(X) \,  \backslash \,  \bY(X) \simlgr
 \ker\Bigl( H^1\big(X,\bN_\mathfrak{G}(\mathfrak{S})\big) \to  H^1(X,\mathfrak{G}) \Bigr) .
$$

To summarize,
$$
\ker\Bigl( H^1\big(X,\bN_\mathfrak{G}(\mathfrak{S})\big) \to  H^1(X,\mathfrak{G}) \Bigr)
$$
classifies the $\mathfrak{G}(X)$--conjugacy classes of
subtori $\mathfrak{S}'$ of $\mathfrak{G}$ which are locally-\'etale $\mathfrak{G}$-conjugate to
$\mathfrak{S}$. For each such $\mathfrak{S}'$,  we denote by
 $\gamma(\mathfrak{S}') \in H^1\big(X,\bN_\mathfrak{G}(\mathfrak{S})\big)$
its $\mathfrak{G}(X)$-conjugacy class.
In terms of torsors, 
$\gamma(\mathfrak{S}')$ is the class of the strict transporter
 ${\bf \rm Transpstr}_{\mathfrak{G}}(\mathfrak{S}, \mathfrak{S}')$  which is  an $\bN_\mathfrak{G}(\mathfrak{S})$-torsor \cite{Gi}.
 We recall here its  definition: for each $X'\to X$
$$
{\bf \rm Transpstr}_{\mathfrak{G}}(\mathfrak{S}, \mathfrak{S}')(X') = \Bigl\{ g \in \mathfrak{G}(X') \, \mid \,
g \, \mathfrak{S}(X'') \, g^{-1}=\mathfrak{S}'(X'') \enskip \, \forall\, 
X'' \to X' \Bigr\}.
$$
The converse map is given as follows  (\cite[Lemme 2.1]{Gi}). Consider an  
$\bN_\mathfrak{G}(\mathfrak{S})$-torsor $\bE$ equipped with a
trivialization 
$$
u:  \mathfrak{G} \simlgr \bE \wedge^{\bN_\mathfrak{G}(\mathfrak{S})} \mathfrak{G}.
$$
where $\wedge$ is the contracted product. Then  
$u^{-1}(\bE \wedge^{\bN_\mathfrak{G}(\mathfrak{S})} \mathfrak{S})$ is a 
subtorus of
$\mathfrak{G}$.

Furthermore, each such $\mathfrak{S}'$  is Zariski-locally $\mathfrak{G}$-conjugated
to $\mathfrak{S}$ if and only if  
$$
\gamma(\mathfrak{S}') \in H^1_{Zar}\big(X,\bN_\mathfrak{G}(\mathfrak{S})\big)
 \subset H^1\big(X,\bN_\mathfrak{G}(\mathfrak{S})\big).
 $$
It follows that the set
 $$
\ker\Bigl( H^1_{Zar}\big(X,\bN_\mathfrak{G}(\mathfrak{S})\big) \to  H^1_{Zar}(X,\mathfrak{G}) \Bigr)
$$
 classifies the $\mathfrak{G}(X)$--conjugacy classes of tori which are locally 
 conjugate to $\mathfrak{S}$ for the Zariski topology.

Now we consider  the following exact sequence of $X$-group-schemes
$$
1 \to \bZ_\mathfrak{G}(\mathfrak{S}) \to \bN_\mathfrak{G}(\mathfrak{S}) \to \bW(\mathfrak{S}) \to 1 .
$$
There is a natural action of  $\bW(\mathfrak{S})(X)$ on 
$H^1(X, \bZ_{\mathfrak{G}}(\mathfrak{S}))$
and it  induces a bijection \cite[III.3.3.1]{Gi}
$$
H^1\big(X, {\bZ}_{\mathfrak{G}}(\mathfrak{S})\big) \,  / \,  {\bf W}(\mathfrak{S})(X) \simlgr 
\ker\Bigl( \text{\it H}^\text{\,1}\big(X,\bN_\mathfrak{G}(\mathfrak{S})\big) \to  
\text{\it H}^\text{\,1}\big(X,\bW(\mathfrak{S})\big) \Bigr).
$$
We consider then the subset $
H^1_{Zar}( X, {\bZ}_\mathfrak{G}(\mathfrak{S})) \,  / \,  \bW(\mathfrak{S})(X)
$
of 
$$
H^1_{Zar}(X,\bN_\mathfrak{G}(\mathfrak{S})) \subset H^1(X,\bN_\mathfrak{G}(\mathfrak{S})).
$$

\begin{proposition}\label{costanera} Assume that the base scheme $X$ is connected.

\smallskip

\noindent {\rm (1)} The pointed set
$$
\ker\Bigl( H^1\big( X, {\bZ}_\mathfrak{G}(\mathfrak{S})\big) \, / \,  {\bW}(\mathfrak{S})(X)
\to \text{\it H}^\text{\,1}( X, \mathfrak{G}) \Bigr)
$$
classifies split subtori of $\mathfrak{G}$ which are
locally conjugate to $\mathfrak{S}$ for the \'etale topology.

\smallskip

\noindent {\rm (2) } We have $$ \ker\Bigl( H^1_{Zar}\big(X, {\bZ}_\mathfrak{G}(\mathfrak{S})\big)
\to \text{\it H}^\text{\,1}_{\rm \it Zar}( X, \mathfrak{G}) \Bigr) \, \, \simlgr \, \,
\ker\Bigl( \text{\it H}^\text{\,1}\big(X, {\bZ}_\mathfrak{G}(\mathfrak{S})\big)
\to\text{\it H}^\text{\,1}(X, \mathfrak{G}) \Bigr).
$$

\noindent {\rm (3)}  The pointed set
$$
\ker\Bigl( H^1_{Zar}\big(X, {\bZ}_\mathfrak{G}(\mathfrak{S})\big) \, / \,  \bW(\mathfrak{S})(X)
\to \text{\it H}^\text{\,1}_{\rm \it Zar}( X, \mathfrak{G}) \Bigr)
$$
classifies $\mathfrak{G}(X)$-conjugacy classes of split subtori of $\mathfrak{G}$ which are
locally conjugate to $\mathfrak{S}$ for the Zariski topology or  equivalently locally for the
\'etale topology.
\end{proposition}

The connectedness assumption is used to ensure that $\bW(\mathfrak{S})$ is
a finite  constant $X$-group (Lemma \ref{critical}). The proof  of the proposition will be given below. It uses the fact that $\bW(\mathfrak{S})$ is
an open subgroup of the constant $X$-group $\bAut(\mathfrak{S})$.

\begin{lemma}\label{compat}
{\rm (1)} The map $H^1\big(X, \bW(\mathfrak{S})\big) \to H^1\big(X,\bAut(\mathfrak{S})\big)$ has  trivial kernel.

\smallskip

\noindent {\rm (2)} The compositum
$$
\big(\mathfrak{G}/\bN_{\mathfrak{G}}(\mathfrak{S})\big)(X) 
\longrightarrow H^1\big(X,\bN_\mathfrak{G}(\mathfrak{S})\big) \longrightarrow 
H^1\big(X,\bW(\mathfrak{S})\big)  \longrightarrow
H^1\big(X, \bAut(\mathfrak{S})\big)
$$
maps   a   subtorus $\mathfrak{S}'$ of $\mathfrak{G}$ which is  \'etale-locally 
conjugate to $\mathfrak{S}$
to the class of the $\bAut(\mathfrak{S})$-torsor ${\bf {\rm Isom}}_{gr}(\mathfrak{S}, \mathfrak{S}')$.
\end{lemma}

\begin{proof}
(1) 
Since $\mathfrak{S}$ is split,
$\bAut(\mathfrak{S})$ is the constant $X$--group associated to ${\rm GL}_r(\mathbb{Z})$.
Hence $\bW(\mathfrak{S})$  is the constant $\X$--group associated to
 a finite subgroup $\Gamma$ of ${\rm GL}_r(\mathbb{Z})$.
We consider an exact sequence of \'etale $\b\X$-sheaves
$$
1 \to {\bf \Gamma}_X \to \GL_r(\mathbb{Z})_X \to (\GL_r(\mathbb{Z}) / \Gamma)_X \to 1.
$$
It gives rise to a long exact sequence of pointed sets
$$
1 \to \Gamma^{\pi_0(X)} \to \GL_r(\mathbb{Z})^{\pi_0( X)} \to
(\GL_r(\mathbb{Z})/ \Gamma)^{\pi_0(X)} \to H^1(X, {\bf \Gamma}_X)
\stackrel{\lambda}{\to} H^1(X, \GL_r(\mathbb{Z})_X).
$$
It follows that $\lambda$ has trivial kernel, because
$\pi_0(X)$
acts trivially on $\Gamma,\GL_r(\mathbb{Z})$ and $\GL_r(\mathbb{Z})/\Gamma$.

\smallskip

\noindent (2) The $\mathfrak{G}(X)$--conjugacy class of $\mathfrak{S}'$
is nothing but the class of the strict transporter
 ${\bf \rm Transpstr}_{\mathfrak{G}}(\mathfrak{S}, \mathfrak{S}')$  which is  an 
 $\bN_\mathfrak{G}(\mathfrak{S})$-torsor.
The change of groups $\bN_\mathfrak{G}(\mathfrak{S}) \to \bW(\mathfrak{S}) \to \bAut(\mathfrak{S})$
applied to that transporter yields indeed
  the $\bAut(\mathfrak{S})$-torsor ${\bf \rm Isom}_{gr}(\mathfrak{S}, \mathfrak{S}')$.
\end{proof}

We now proceed with the proof of the Proposition.
\begin{proof}

\smallskip

\smallskip
\noindent (1)
Consider the diagram of exact sequences of pointed sets
$$
\begin{CD}
&&  \big(\mathfrak{G}/\bN_\mathfrak{G}(\mathfrak{S})\big)(X) & &1 \\
&&@V{\varphi}VV@VVV \\
H^1\big(X,{\bZ}_\mathfrak{G}(\mathfrak{S})\big) @>{i_*}>> H^1\big(X,\bN_\mathfrak{G}(\mathfrak{S})\big) @>>>  H^1\big(X,\bW(\mathfrak{S})\big) \\
&&&& @VVV \\
&&&& H^1(X, \bAut(\mathfrak{S})) . \\
\end{CD}
$$
Let $\mathfrak{S}'$ be a subtorus of $\mathfrak{G}$ which is \'etale-locally
conjugate to $\mathfrak{S}$.  According to  Lemma \ref{compat}.(2),
$\mathfrak{S}'$ is split if and only if the image of $\gamma( \mathfrak{S}')$
in  $H^1\big(X, \bAut(\mathfrak{S})\big)$ vanishes. Since
$H^1\big(X,\bW(\mathfrak{S})\big) \to  H^1\big(X, \bAut(\mathfrak{S})\big)$ has trivial kernel by 
the first part of
that Lemma, it follows that
$\mathfrak{S}'$ is split if and only if the image of $\gamma( \mathfrak{S}')$
in  $H^1\big(X, \bW(\mathfrak{S})\big)$ vanishes, that is if and only if
 $\gamma( \mathfrak{S}')$ admits a reduction to ${\bZ}_\mathfrak{G}(\mathfrak{S})$.

 \smallskip
 \noindent (2) The statement is local so we may assume that $X=\Spec(R)$
with $R$ a local ring. We need to show that
the kernel $\ker\Bigl( H^1\big( R, \bZ_\mathfrak{G}(\mathfrak{S})\big)
\to \text{\it H}^\text{\,1}( {\rm \it R}, \mathfrak{G}) \Bigr)$ vanishes.
But this is \cite[XXVI 5.10.(i)]{SGA3}.

\smallskip

\noindent (3) By (2), the
pointed set $$
\ker\Bigl( H^1_{Zar}\big( X, {\bZ}_\mathfrak{G}(\mathfrak{S})\big) / \bW(\mathfrak{S})(X)
\to \text{\it H}^\text{\,1}_{\rm \it Zar}( X, \mathfrak{G}) \Bigr)
$$
maps bijectively to
$\ker\Bigl( H^1( X, {\bZ}_\mathfrak{G}(\mathfrak{S})) / \bW(\mathfrak{S})(X)
$ $\to H^1( X, \mathfrak{G})  \Bigr)$
so it classifies $\mathfrak{G}(X)$-conjugacy classes of split subtori of $\mathfrak{G}$ which are
\'etale-locally conjugated to $\mathfrak{S}$  by (2). But for such a subtorus $\mathfrak{S}'$, 
its conjugacy class
$\gamma( \mathfrak{S}')$ belongs then to $H^1_{Zar}(X, \bN_\mathfrak{G}(\mathfrak{S}))$ 
so that
$\mathfrak{S}'$ is locally conjugated to $\mathfrak{S}$ for the Zariski topology.
\end{proof}

The following strengthens Proposition 10.1 of \cite{CGP}.

\begin{proposition} Assume that $X$ is a connected scheme
and $\mathfrak{S}$ is  a split subtorus  of  $\mathfrak{G}$
such that $\mathfrak{S}$ is generically maximal split in $\mathfrak{G}$.

\smallskip

\noindent {\rm (1)} The torus $\mathfrak{S}$ is a maximal split subtorus of $\mathfrak{G}$ and
${\bZ}_\mathfrak{G}(\mathfrak{S})$ is a Levi subgroup  of a parabolic  subgroup
of $\mathfrak{G}$ which is generically minimal.

\smallskip

\noindent {\rm (2)} The generically maximal split split subtori of $\mathfrak{G}$
are the subtori locally conjugated to $\mathfrak{S}$ for the Zariski topology.
Their $\mathfrak{G}(X)$-conjugacy class  are classified by the pointed set
 $\ker\Big( H^1_{Zar}\big( X, {\bZ}_\mathfrak{G}(\mathfrak{S})\big) / \bW(\mathfrak{S})(X)$ $ \to 
 H^1_{Zar}\big(X, \mathfrak{G}\big)\Big)$.
\end{proposition}

\begin{proof}
(1) The first assertion is obvious.
It is known that $\bZ_\mathfrak{G}(\mathfrak{S})$ is a Levi subgroup of   a parabolic subgroup
$\bP$ of $\mathfrak{G}$ \cite[XXVI.6.2]{SGA3}. Since  $\mathfrak{S}$ remains maximal 
split at the generic point of
$X$, it follows that  $\bP \times_X \kappa(X)$ is a minimal parabolic subgroup of 
$\mathfrak{G}_{\kappa(X)}$ 
\cite[XXVI.6]{SGA3}.

\smallskip

\noindent $(2)$ If $\mathfrak{S}'$ is a split torus of $\mathfrak{G}$ which
is generically maximal split then it is Zariski-locally maximal split.
In particular, $\mathfrak{S}'$ is Zariski-locally $\mathfrak{G}$-conjugate to $\mathfrak{S}$ by
Demazure's conjugacy theorem \cite[XXVI.6.16]{SGA3}.
Conversely, if $\mathfrak{S}'$ is  $\mathfrak{G}$-conjugate to $\mathfrak{S}$
in the Zariski topology, it is generically maximal split because $\mathfrak{S}$ is.
The second assertion follows from Proposition \ref{costanera}(3).
\end{proof}

We apply the above general considerations to the particular case of $X = \Spec(R')$,
 a simple adjoint group scheme $\mathfrak{G}$ over $R'$ and a split torus $\mathfrak{S}$
 of $\mathfrak{G}$ which is generically maximal split. 
 By Proposition~\ref{vanish} we have $H^1(X,\mathfrak{G})=1.$ Thus
 $H^1_{Zar}(X,\mathfrak{G})=1$ and therefore
 $$
 \ker \Bigl( H^1_{Zar}\big(X, {\bZ}_\mathfrak{G}(\mathfrak{S})\big) \to 
 H^1_{Zar}\big(X, \mathfrak{G})\big)\Bigr)=
 H^1_{Zar}\big( X, \bZ_\mathfrak{G}(\mathfrak{S})\big).
 $$ 
 Since  by Corollary~\ref{corvanish} our groups are always quasi-split, the centralizer 
 $\widetilde{\mathfrak{S}}={\bZ}_\mathfrak{G}(\mathfrak{S})$ is always a maximal torus  of 
 $\mathfrak{G}.$ Since $\mathfrak{G}$ is of adjoint type, 
 $\widetilde{\mathfrak{S}}$ is 
 the direct product of a split $R'$-torus and
 a Weil 
 restriction $R_{S'/R'}(G_{m,S'})$  of a one-dimensional split $S'$-torus 
 where $S'/R'$ is an \'etale extension of 
 $R'$ of degree $2$ or $3$.

By Shapiro's lemma $H^1(R', \widetilde{\mathfrak{S}})$ vanishes whenever $S'$ is 
isomorphic to the coordinate ring of the affine line minus a finite number of points. This is 
the case whenever our dessin is of genus $0$, namely in all cases except for case $(c,c)$ in 
type $D_4.$ In this exceptional case the situation is quite the opposite. The scheme $S'$ 
is a non-empty 
open affine subscheme of an elliptic curve over an algebraically closed field, 
hence $H^1(S',G_{m,S'})=H^1_{Zar}(\X,\widetilde{\mathfrak{S}})$ is infinite.
It follows that there are  
infinitely many conjugacy classes of maximal split tori inside the corresponding twisted adjoint  
$R'$--group $\mathfrak{G}$ of type $D_4$.

\section{Conjugacy questions for three-point Lie algebras}

Let $\CL$ be a twisted form of $\gg \otimes_k R'$. The role of split Cartan subalgebras 
for the infinite dimensional $k$-Lie algebra $\CL$ is played by maximal abelian
$k$-diagonalizable subalgebras. Recall that a subalgebra $\bk$ of
$\CL$ is $k${\it -diagonalizable} if there exists a
$k$-basis $v_\lambda  ,$ $\lambda  \in \Lambda $, of 
$\CL$ consisting of eigenvectors for the adjoint action of $\bk$,
i.e.
for all $\lambda  \in \Lambda  $ and $x\in\bk,$ there exists
$\lambda  (x)\in k$ such that $[x,v_\lambda  ]=\lambda  (x)v_\lambda
 .$
\medskip
A $k$-diagonalizable subalgebra of 
$\CL$ is
necessarily abelian.  The MADs are those $k$-diagonalizable
subalgebras which are maximal with respect to inclusion.

Using the correspondence between maximal split $R'$-tori of $\mathfrak{G}$ and 
MADs of $\CL = {\rm Lie}(\mathfrak{G})$ given in  \cite[Theorem 7.1]{CGP} and an 
obvious fact that a maximal split
torus in $\mathfrak{G}$ is generically maximal split (because by Corollary~\ref{corvanish} 
there are no 
$R'$-anisotropic semisimple group schemes) 
we obtain the following.

\begin{theorem} 
If $\CL$ is not of dessin type $(c,c)$ there exists a single conjugacy class of MADs under the adjoint action of $\mathfrak{G}(R)$ on $\CL.$ If $\CL$ is of type $(c,c)$ the number of conjugacy classes of MADs is infinite. \qed
\end{theorem}






\begin{thebibliography}{99}












\bibitem[B]{B} M. Bremner, Generalized affine Kac-Moody Lie algebras
over localizations of the polynomial ring in one variable, 
Canadian Math. Bull. {\bf 37} (1994), 21-28

\bibitem[BT]{BT} G. Benkart and P. Terwilliger, {\it The equitable basis for $s\ell_2$},
Math. Z. {\bf 268} (2011), 536--557.



\bibitem[CGP]{CGP} V. Chernousov, P. Gille and A. Pianzola, {\it
Conjugacy theorems for loop reductive group schemes and Lie algebras}, Bulletin of Mathematical Sciences (to appear).

\bibitem[DG]{DG} M. Demazure et P. Gabriel,
{\it Groupes alg\'ebriques}, Masson (1970).

\bibitem[EGAIV]{EGAIV} A. Grothendieck (avec la collaboration de J.
Dieudonn\'e), {\it El\'ements de G\'eom\'etrie Alg\'ebrique IV},
Publications math\'ematiques de l'I.H.\'E.S. no 20, 24, 28 and 32
(1964 - 1967).

\bibitem[E]{E} A. Elduque, The $S_4$-action and the Tetrahedraon
algebra, Proc. Royal Soc. Edinburgh 137A (2007), 1227-1248.

\bibitem[EF]{EF} Etingof, P.I., Frenkel, I.B., Central extensions of current groups in two 
dimensions. Commun. Math. Phys. {\bf 165} (1994), 429--444.

\bibitem[GP]{GP} P.~Gille and A.~Pianzola,
{\it Galois cohomology and forms of algebras over Laurent polynomial rings},
Math. Ann. {\bf 338} (2007),  497--543.


\bibitem[G]{G} P. Gille, {\it Sur la classification des sch\'emas
 en groupes semi-simples},  preprint (2011), author's URL.

\bibitem[GP]{GP1} P. Gille and A. Pianzola, {\it Galois cohomology and 
forms of algebras over Laurent polynomial rings},  Mathematische Annalen  
{\bf 338} (2007), 497-543.


\bibitem[Gi]{Gi} J. Giraud, {\it Cohomologie non ab\'elienne},  Die Grundlehren der
mathematischen Wissenschaften {\bf 179} (1971), Springer-Verlag.



\bibitem[H]{H1} G. Harder, {\it Halbeinfache
Gruppenschemata \"uber Dedekindringen}, Invent. Math. {\bf 4}
(1967), 165--191.


\bibitem[Hart]{Hart} R. Hartshorne,  {\it Algebraic Geometry}, Graduate text in Mathematics 
{\bf 52}, Springer-Verlag.

 \bibitem[HT]{HT} B. Hartwig and P. Terwilliger, {\it The Tetrahdron
algebra, the Onsager algebra, and the $s\ell_2$ loop algebra},
Jour. of Algebra {\bf 308} (2007), 840--863.

\bibitem[HS]{HS}F. Herrlich, G. Schmith\"usen, Handbook of Teichm\"uller theory,
Vol. II, 767-809, IRMA Lect. Math. Theor. Phys., 13, Eur. Math. Soc, Z\"urich, 2009.

\bibitem[Kac]{Kac} V.~Kac,
\emph{Infinite dimensional Lie algebras}, third edition,
Cambridge University Press, Cambridge, 1990.

\bibitem[LZ]{LZ} S. K. Lando and A. K. Zvonkin, Graphs on surfaces and their
applications, Encyclopaedia of Mathematical Sciences, {\bf 141} (II), Springer (2004).  



\bibitem[M]{M2} J.S. Milne, { \it Lectures on Etale cohomology},
home page of the author.


\bibitem[P]{P1} A. Pianzola, {\it Vanishing of $H^1$ for Dedekind rings 
and applications to loop
algebras}, C. R. Acad. Sci. Paris, Ser. I {\bf 340} (2005), 633-638.


\bibitem[P2]{P2} A.~Pianzola, \emph{Derivations of certain algebras
    defined  by     $\acute{e}$tale descent}. Math. Z. \textbf{264} (2010).

\bibitem[PS]{PS} A. Presley and G. Segal, {\it  Loop groups}, Oxford
Mathematical Monographs, Clarendon Press 1986.

\bibitem[Sch]{Sch} L. Schneps, Dessins d'enfants on the Riemann sphere,
London
Math. Soc. Lect. Note Ser. {\bf 200} (1994), 47-77, Cambridge University Press.

 \bibitem[Se]{Se} J.-P.~Serre, {\it Galois Cohomology},
Springer, 1997.

\bibitem[St]{St} R. Steinberg, {\it  Regular elements of
semi-simple algebraic groups},
Publications math. IHES  {\bf 25} (1965), 49-80.

\bibitem[Sz]{Sz} T. Szamuely, {Galois groups and Fundamental groups}, Cambridge 
Studies in Advance Mathematics {\bf 117} Cambridge Univ. Press (2009).



\bibitem[SGA1]{SGA1} {\it S\'eminaire de G\'eom\'etrie alg\'ebrique de
l'I.H.E.S.,  Rev\^etements \'etales et groupe fondamental, dirig\'e
par  A. Grothendieck},  Lecture Notes in Math. 224. Springer (1971).

\bibitem[SGA3]{SGA3} {\it S\'eminaire de G\'eom\'etrie alg\'ebrique de
l'I.H.E.S., 1963-1964, sch\'emas en groupes, dirig\'e par M.
Demazure et A. Grothendieck},  Lecture Notes in Math. 151-153.
Springer (1970).

\bibitem[W]{W} J. Wolfart, Kinderzeichnungen und Uniformisierung, Unpublished 
manuscript (2001).

\bibitem[Z]{Z} L. Zapponi, What is a Dessin d'enfant? Notices of AMS {\bf 50} (2003), n. 7, 788-789.
\end{thebibliography}
\end{document}